%% file: OnSzabadosConj.tex
\documentclass[10pt,reqno]{amsart}
\usepackage[british]{babel}
\usepackage{amsmath,amsfonts,amssymb,amstext,amsbsy,amsopn,amsthm,amsxtra,amscd} 
\usepackage{graphicx}
\usepackage{pdfpages}
\usepackage{caption}
\usepackage{subcaption}
\usepackage{color}
\usepackage{hyperref}
\usepackage{xcolor}
\usepackage{enumerate}
\usepackage{textcomp}

\numberwithin{equation}{section}

\newtheorem{theorem}{Theorem}[section]
\newtheorem{proposition}[theorem]{Proposition}
\newtheorem{corollary}[theorem]{Corollary}
\newtheorem{lemma}[theorem]{Lemma}
\newtheorem{definition}[theorem]{Definition}
\newtheorem{remark}[theorem]{Remark}
\newtheorem{claim}[theorem]{Claim}

\newtheorem{conjecture}[theorem]{Conjecture}

\newcommand{\bb}[1]{\mathbb{#1}}

\DeclareMathOperator{\supp}{Supp}
\DeclareMathOperator{\conv}{Conv}
\DeclareMathOperator{\ann}{Ann}

\DeclareMathOperator{\dist}{Dist}
\DeclareMathOperator{\nexpl}{NEL}
\DeclareMathOperator{\nexpd}{ONED}

\newcommand{\sob}[1]{\ensuremath{{\mathbin |}\raise-.4ex\hbox{$#1$}}}

\title{ON PERIODIC DECOMPOSITIONS AND NONEXPANSIVE LINES}
\author{Cleber Fernando Colle}
\address{Rio de Janeiro State University, Rio de Janeiro, Brazil}
\email{cleber.colle@ime.uerj.br}


\begin{document}
\begin{abstract}
In his Ph.D. thesis, Michal Szabados conjectured that for a not fully periodic configuration with a min\-i\-mal periodic decomposition the nonexpansive lines are exactly the lines that contain a period for some periodic configuration in such decomposition. In this paper, we study Szabados's conjecture. First, we show that we may consider a minimal periodic decomposition where each periodic configuration is defined on a finite alphabet. Then we prove that Szabados's conjecture holds for a wide class of configurations, which includes all not fully periodic low convex pattern complexity configurations.  
\end{abstract}

\maketitle 

\section{Introduction} 

In his Ph.D. thesis \cite{Szabadosthesis}, with an algebraic viewpoint on symbolic configurations, Michal Szabados showed that every low pattern complexity configuration can be decomposed into a finite sum of periodic configurations. Such a periodic decomposition theorem is the heart of new recent results related to Nivat's conjecture and other topics \cite{kari,Szabados,Moutot,Moutot19,KariSzabados,Moutot21,Colle22}. If we consider a not fully periodic configuration with a minimal periodic decomposition, i.e., a periodic decomposition with the smallest possible number of periodic configurations, a nonexpansive line contains a period of some periodic configuration. Michal Szabados conjectured that the conversely also holds.

To make the statements precise we need to introduce some terminologies.

\subsection{Symbolic dynamics}

Let $\mathcal{A}$ be an alphabet with at least two elements. The elements of $\mathcal{A}^{\bb{Z}^2}$, called \emph{configurations}, have the form $\eta = (\eta_{g})_{g \in \bb{Z}^2}$, where $\eta_g \in \mathcal{A}$ for all $g \in \bb{Z}^2$. For each \(u \in \bb{Z}^2\), the \emph{shift application} \(T^u : \mathcal{A}^{\bb{Z}^2} \longrightarrow \mathcal{A}^{\bb{Z}^2}\) is defined by \((T^u\eta)_g = \eta_{g-u}\) for all \(g \in \bb{Z}^2\) and all \(\eta \in \mathcal{A}^{\bb{Z}^2}\). A configuration $\eta \in \mathcal{A}^{\bb{Z}^2}$ is said to be \emph{periodic} if there exists a non-zero vector $h \in \bb{Z}^2$, called \emph{period} of $\eta$, such that \(T^h\eta = \eta\). If \(\eta\) has two periods linearly independents over \(\bb{R}^2\), we say that \(\eta\) is \emph{fully periodic}. 

Let \(Orb \, (\eta) = \{T^u\eta : u \in \bb{Z}^2\}\) denotes the \emph{\(\bb{Z}^2\)-orbit of \(\eta \in \mathcal{A}^{\bb{Z}^2}\)}. If \(\mathcal{A}\) is a finite alphabet endowed with the discrete topology, then $\mathcal{A}^{\bb{Z}^2}$ equipped with the product  topology is a metrizable compact space. In particular, for all \(\eta \in \mathcal{A}^{\bb{Z}^2}\), the set \(\overline{Orb \, (\eta)}\), where the bar denotes the closure, is a compact \emph{subshift}, i.e., a closed sub\-set invariant by the $\bb{Z}^2$-action $T^u$, $u \in \bb{Z}^2$.

If \(\mathcal{A}\) is finite, we say that a configuration \(\eta \in \mathcal{A}^{\bb{Z}^2}\) has \emph{low pattern complexity} if \(|\{T^u\eta\sob{\mathcal{S}} : u \in \bb{Z}^2\}| \leq |\mathcal{S}|\) holds for some non-empty, finite set \(\mathcal{S} \subset \bb{Z}^2\), where $\cdot \sob{\mathcal{S}}$ means the restriction to the set $\mathcal{S}$. If in addition $\mathcal{S}$ is \emph{convex}, i.e., a subset of $\bb{Z}^2$ whose convex hull in $\bb{R}^2$, denoted by $\conv(\mathcal{S})$, is closed and $\mathcal{S} = \conv(\mathcal{S}) \cap \bb{Z}^2$, we say that \(\eta\) has \emph{low convex pattern complexity}.   

From now on, we will assume that $\mathcal{A}$ is a finite alphabet endowed with the dis- cre\-te topology.

\subsection{Nonexpansive lines}

Let $\ell \subset \bb{R}^2$ be a line through the origin. Given \(t > 0\), the $t$-neighbourhood of $\ell$ is defined as \[\ell^{t} = \{g \in \bb{Z}^2 : \dist(g,\ell) \leq t\}.\] Following Boyle and Lind \cite{boyle}, we say that $\ell$ is an \emph{expansive} line on \(\overline{Orb \, (\eta)}\) if there exists $t>0$ such that \[\forall \ x,y \in \overline{Orb \, (\eta)}, \quad x\sob{\ell^{t}} = y\sob{\ell^{t}} \implies x = y.\] Other\-wise, $\ell$ is called a \emph{nonexpansive} line on \(\overline{Orb \, (\eta)}\). 

For a configuration $\eta \in \mathcal{A}^{\bb{Z}^2}$\!, we use $\nexpl(\eta)$ to denote the set for\-med by the nonexpansive lines on \(\overline{Orb \, (\eta)}\). A particular case of Boyle-Lind Theorem \cite[Theorem~3.7]{boyle} implies that \(\nexpl(\eta)\) has at least one element if the subshift $\overline{Orb \, (\eta)}$ is infinite. An immediate application of Boyle-Lind Theorem allows us to state the following result:

\begin{theorem}\label{corollary_Boyle-Lind_tmh}
Let $\eta \in \mathcal{A}^{\bb{Z}^2}$ be a configuration. Then \(\nexpl(\eta) = \emptyset\) if and only if \(\eta\) is fully periodic.
\end{theorem}  

When $\eta \in \mathcal{A}^{\bb{Z}^2}$ is a periodic configuration, but not fully periodic, it is easy to see that the unique line $\ell \in \nexpl(\eta)$, which exists due to Boyle-Lind Theorem, contains any period for $\eta$. Hence, any configuration $\eta \in \mathcal{A}^{\bb{Z}^2}$ where $\nexpl(\eta)$ has at least two lines can not be periodic.

\subsection{The algebraic viewpoint}

In \cite{KariSzabados}, Kari and Szabados introduced an algebraic viewpoint on symbolic configurations. To follow their approach, we consider \(R = \bb{Z}\) or some finite field, and we represent any configuration \(\eta \in R^{\bb{Z}^2}\) as a formal power series over two variable \(x_1\) and \(x_2\) with coefficients \(\eta_g \in R\), i.e., as an element of \[R[[X^{\pm 1}]] = \left\{\sum_{g \in \bb{Z}^2} a_gX^g : a_g \in R\right\},\] where $g = (g_1,g_2) $ and $X^g$ is a shorthand for $x_1^{g_1}x_2^{g_2}$. Let \(R[X^{\pm 1}] \subset R[[X^{\pm 1}]]\) denote the set of Laurent polynomials with coefficients in \(R\). Given a Laurent polynomial $\varphi(X) = a_1X^{u_1}+\cdots+a_nX^{u_n}$, with $a_i \in R$ and $u_i \in \bb{Z}^2$, and a configuration $\eta \in R[[X^{\pm 1}]]$, we may consider the configuration \(\varphi \eta \in R[[X^{\pm 1}]]\), where
\begin{equation}\label{defproduct}
(\varphi \eta)_g = a_1\eta_{g-u_1}+\cdots+a_n\eta_{g-u_n} \quad \forall g \in \bb{Z}^2.
\end{equation}      
Let $\eta \in \mathcal{A}^{\bb{Z}^2}$, with $\mathcal{A} \subset R$, be a configuration. A Laurent polynomial $\varphi \in R[X^{\pm 1}]$ \emph{annihilates $\eta$} if and only if \(\varphi \eta = 0\), where by \(0\) we mean the zero configuration. The set of Laurent polynomials $\varphi \in R[X^{\pm 1}]$ that annihilates $\eta \in \mathcal{A}^{\bb{Z}^2}$\! is de\-no\-ted by $\ann_R(\eta)$. In this algebraic setting, $\eta \in R[[X^{\pm 1}]]$ is periodic of period $h \in \bb{Z}^2$ if and only if $(X^{ h}-1)\eta = 0$. 

We remark that the operations in (\ref{defproduct}) are the binary operations of $(R,+,\cdot)$. The multiplicative and additive identities on $R$ are denoted by $1$ and $0$, respectively, and, for any $x,y \in R$, $x-y := x+(-y)$, where $-y \in R$ denotes the inverse additive of $y$. In the most part of this work, we will consider \(R = \bb{Z}_p\) for some prime number \(p \in \bb{N}\footnote{We use $\bb{N} = \{1,2,\ldots\}$ and $\bb{Z}_+ = \bb{N} \cup \{0\}$.}\).

\begin{theorem}[Kari and Szabados \cite{KariSzabados}]\label{theorKS}
Let $\eta \in \mathcal{A}^{\bb{Z}^2}$, with $\mathcal{A} \subset \bb{Z}$, be a low pattern complexity configuration. Then there exist periodic configurations \(\eta_1, \ldots, \eta_m \in \bb{Z}[[X^{\pm 1}]]\) such that \(\eta = \eta_1+\cdots+\eta_m\).
\end{theorem} 

The configurations $\eta_i \in \bb{Z}[[X^{\pm 1}]]$ in the Theorem \ref{theorKS} may be defined on infinite alphabets (see \cite[Example 17]{KariSzabados}).

Let $\eta \in \mathcal{A}^{\bb{Z}^2}$, with \(\mathcal{A} \subset R\), and suppose $\eta_1, \ldots, \eta_m \in R[[X^{\pm 1}]]$ are periodic configurations such that $\eta = \eta_1+ \cdots+\eta_m$. We call $\eta = \eta_1+ \cdots+\eta_m$ a \emph{$R$-periodic decomposition}. If $m \leq n$ for every $R$-pe\-ri\-o\-dic de\-com\-po\-si\-tion $\eta = \eta'_1+ \cdots+\eta'_{n}$, we call $\eta = \eta_1+\cdots+\eta_m$ a \emph{$R$-mi\-ni\-mal periodic decomposition} and the number $m$ is called the \emph{\(R\)-order of $\eta$}.

\begin{remark}\label{rem_vectorminimalperioddecomposit}
Let $\eta \in \mathcal{A}^{\bb{Z}^2}$, with \(\mathcal{A} \subset R\), and suppose $\eta = \eta_1+ \cdots+\eta_m$ is a \(R\)-minimal periodic decomposition.
\begin{enumerate}[(i)]\setlength{\itemsep}{6pt}
	\item If \(\eta\) has \(R\)-order \(m \geq 2\), then any two periods for $\eta_i$ and $\eta_j$, with $i \neq j$, are in distinct directions;
	
	\item If \(\eta\) has \(R\)-order \(m = 1\), then either \(\eta\) periodic or \(\eta\) is fully periodic,
\end{enumerate}
\end{remark}

\begin{remark}
If $\eta = \eta_1+ \cdots+\eta_m$ is a \(\bb{Z}_p\)-minimal periodic decomposition, then each configuration \(\eta_i \in \bb{Z}_p[[X^{\pm 1}]]\), with \(1 \leq i \leq m\), is defined on a finite alphabet.  
\end{remark}

\medbreak

If $\eta = \eta_1+\cdots+\eta_m$ is a $\bb{Z}$-minimal periodic decomposition, it is easy to see that every nonexpansive line on \(\overline{Orb \, (\eta)}\) contains a period of some $\eta_i$, with $1 \leq i \leq m$ (see Remark~\ref{rem_converSazabados'sconj}). In his Ph.D. thesis \cite{Szabadosthesis}, Michal Szabados conjectured that the conversely also holds:

\begin{conjecture}[Szabados]\label{Szabados'sconj}
Let $\eta \in \mathcal{A}^{\bb{Z}^2}$, with $\mathcal{A} \subset \bb{Z}$, be a configuration. If $\eta$ is not fully periodic and $\eta = \eta_1+\cdots+\eta_m$ is a $\bb{Z}$-minimal periodic decomposition, then, for each $1 \leq i \leq m$, a line $\ell \subset \bb{R}^2$ contains a period for $\eta_i$ if and only if $\ell \in \nexpl(\eta)$.
\end{conjecture}

Szabados's conjecture is a very recent open problem in symbolic dynamics and for the best we know there are not many works toward a proof for the conjecture. 

Thanks to Proposition~\ref{prop_expan_period}, we know that Szabados's conjecture holds for any not fully periodic configuration with \(\bb{Z}\)-order \(2\). In \cite{Colle22}, the author showed that Szabados's conjecture holds for a not fully periodic low convex pattern complexity configuration $\eta \in \mathcal{A}^{\bb{Z}^2}$, with $\mathcal{A} \subset \bb{Z}$, such that all non-periodic configurations in \(\overline{Orb \, (\eta)}\) have the same \(\bb{Z}\)-order. As a corollary, the author got that Szabados's conjecture holds for any not fully periodic low convex pattern complexity configuration with \(\bb{Z}\)-order \(3\).

In this paper, as a corollary of our main result we get that Szabados's conjecture holds for any not fully periodic low convex pattern complexity configuration.  

\subsection{Our contributions}

We believe that the difficulty in approaching Szabados's conjecture remains mostly on the fact that the periodic configurations in a periodic decomposition may not be defined on finite alphabets. Our first result allows us to get around this difficulty.

Given \(n \in \bb{N}\), let us denote \([[n]] = \{0,1,2, \ldots, n-1\}\). 

\begin{theorem}\label{mainthm1}
Let $\eta \in \mathcal{A}^{\bb{Z}^2}$\!, with \(\mathcal{A} \subset \bb{Z}_+\), be a configuration. If $\eta = \eta_1+\cdots+\eta_m$ is a \(\bb{Z}\)-mi\-ni\-mal periodic decomposition, then, for some prime number \(p \in \bb{N}\), with \(\mathcal{A} \subset [[p]]\), $\overline{\eta} = \overline{\eta}_1+\cdots+\overline{\eta}_m$ is a \(\bb{Z}_p\)-mini\-mal periodic decomposition, where the bar denotes the congruence modulo \(p\).
\end{theorem}

By using Theorem \ref{mainthm1}, we get the following result:

\begin{theorem}\label{mainthm2}
Let $\eta \in \mathcal{A}^{\bb{Z}^2}$, with \(\mathcal{A} \subset \bb{Z}\), be a not fully periodic configuration and suppose \(\eta = \eta_1+\cdots+\eta_m\) is a \(\bb{Z}\)-minimal periodic decomposition. If for any configuration \(x \in \overline{Orb \, (\eta)}\) the subshift \(\overline{Orb \, (x)}\) has a periodic configuration, then, for each $1 \leq i \leq m$, a line $\ell \subset \bb{R}^2$ contains a period for $\eta_i$ if and only if $\ell \in \nexpl(\eta)$.
\end{theorem}

As an application of Theorem \ref{mainthm2}, we get the following result:

\begin{corollary}
Let $\eta \in \mathcal{A}^{\bb{Z}^2}$, with \(\mathcal{A} \subset \bb{Z}\), be a not fully periodic low convex pattern\linebreak complexity configuration and suppose \(\eta = \eta_1+\cdots+\eta_m\) is a \(\bb{Z}\)-minimal periodic decomposition. Then, for each $1 \leq i \leq m$, a line $\ell \subset \bb{R}^2$ contains a period for $\eta_i$ if and only if $\ell \in \nexpl(\eta)$.
\end{corollary}
\begin{proof}
Theorem 1.14 of \cite{Colle22} and standard results from \cite{van,KariSzabados} imply that for any con-\linebreak figuration \(x \in \overline{Orb \, (\eta)}\) the subshift \(\overline{Orb \, (x)}\) has a periodic configuration (see Propo\-si\-tions~2.10 and 2.12 of \cite{Colle22} for detailed statements). Therefore, the corollary follows by Theorem \ref{mainthm2}. 
\end{proof} 

\section{Background}
\label{sec2}
In this section, we will revisit terminologies, known results or immediate applications of known results.

\subsection{One-sided nonexpansive directions}

In the sequel, we will revisit a more refined version of expansiveness called one-sided nonexpansiveness and introduced by Cyr and Kra in \cite{van}. We recall that two vectors are parallel if they have the same direction and antiparallel if they have opposite directions. Two oriented objects in $\bb{R}^2$ are said to be \emph{(anti)parallel} if the adjacent vectors to their respective orientations are (anti)parallel, where by object we mean an oriented line, an oriented line segment or a vector. In a slight abuse of notation, we will view an oriented line also as a subset of $\bb{R}^2$.

Given a line $\ell \subset \bb{R}^2$, we use $\pmb{\ell}$ to denote the line $\ell$ endowed of a given o\-ri\-en\-ta\-tion. Let $\pmb{\ell} \subset \bb{R}^2$ be an oriented line. The half plane determined by $\pmb{\ell}$ is defined as \[\mathcal{H}(\pmb{\ell}) := \{(g_1,g_2) \in \bb{Z}^2 : u_1g_2-u_2g_1 \geq 0\},\] where $(u_1,u_2) \in \bb{R}^2$ is a non-zero vector parallel to $\pmb{\ell}$. Let $\eta \in \mathcal{A}^{\bb{Z}^2}$ be a configuration. An oriented line $\pmb{\ell} \subset \bb{R}^2$ through the origin is called a \emph{one-sided expansive direction on \(\overline{Orb \, (\eta)}\)} if \[\forall \ x,y \in \overline{Orb \, (\eta)}, \quad x\sob{\mathcal{H}(\pmb{\ell})} = y\sob{\mathcal{H}(\pmb{\ell})} \ \Longrightarrow \ x = y.\] Otherwise, $\pmb{\ell}$ is called a \emph{one-sided non\-ex\-pan\-si\-ve direction on \(\overline{Orb \, (\eta)}\)}. 

For a configuration $\eta \in \mathcal{A}^{\bb{Z}^2}$, we use $\nexpd(\eta)$ to denote the set formed by the one-sided non\-ex\-pan\-si\-ve di\-rec\-tions on $\overline{Orb \, (\eta)}$. By the compactness of $\overline{Orb \, (\eta)}$, a line $\ell \in \nexpl(\eta)$ if and only if, for some orientation, $\pmb{\ell} \in \nexpd(\eta)$ (see \cite{van} for more details). As highlighted in \cite{van}, lines in $\nexpl(\eta)$ or oriented lines in \(\nexpd(\eta)\) are rational, where by \emph{rational} we mean a line or an oriented line in $\bb{R}^2$ with a rational slope. 

\subsection{Geometric notations}

Let $\pmb{\ell} \subset \bb{R}^2$ be an oriented line and suppose $\mathcal{S} \subset \bb{Z}^2$ is a non-empty, convex set such that $\mathcal{S}+u \subset \mathcal{H}(\pmb{\ell})$ for some $u \in \bb{Z}^2$. The \emph{support line of $\mathcal{S}$} determined by $\pmb{\ell}$, denoted by $\pmb{\ell}_{\mathcal{S}}$, is defined as the oriented line $\pmb{\ell}'$ parallel to $\pmb{\ell}$ such that $\mathcal{S} \subset \mathcal{H}(\pmb{\ell}')$ and $\pmb{\ell}' \cap \mathcal{S} \neq \emptyset$. Of course, if $\mathcal{S} \subset \bb{Z}^2$ is a non-empty, finite, convex set, then $\pmb{\ell}_{\mathcal{S}}$ is defined. 

Let $\mathcal{S} \subset \bb{Z}^2$ be a non-empty, convex set. If \(\mathcal{S}\) has at least two points and \(|\mathcal{S} \cap \pmb{\ell}'_{\mathcal{S}}| = 1\) holds for some oriented line $\pmb{\ell}' \subset \bb{R}^2$, we call \(\mathcal{S} \cap \pmb{\ell}'_{\mathcal{S}}\) a \emph{vertex} of \(\mathcal{S}\). If $\conv(\mathcal{S})$ has positive area and \(|\mathcal{S} \cap \pmb{\ell}'_{\mathcal{S}}| > 1\) holds for some oriented line $\pmb{\ell}' \subset \bb{R}^2$, we call \(\conv(\mathcal{S} \cap \pmb{\ell}'_{\mathcal{S}})\) an \emph{edge} of \(\mathcal{S}\). We use $E(\mathcal{S})$ to denote the set of edges of $\mathcal{S}$.

If $\mathcal{S} \subset \bb{Z}^2$ is a convex set (possibly infinite) such that $\conv(\mathcal{S})$ has positive area, our standard convention is that the boundary of $\conv(\mathcal{S})$ is positively oriented. With this convention, every edge $w \in E(\mathcal{S})$ inherits a natural orientation from the boundary of $\conv(\mathcal{S})$. Moreover, if \(w = \conv(\mathcal{S} \cap \pmb{\ell}'_{\mathcal{S}})\), then \(w\) and \(\pmb{\ell}'\) are parallels.


\subsection{Generating sets}
The notion of generating set was developed in \cite{van} and studied in a more general setting on \cite{frankskra}. 

\begin{definition}
Let $\eta \in \mathcal{A}^{\bb{Z}^2}$ be a configuration and suppose $\mathcal{S} \subset \bb{Z}^2$ is a finite set. A point $g \in \mathcal{S}$ is said to be \emph{$\eta$-generated by $\mathcal{S}$} if $P_{\eta}(\mathcal{S}) = P_{\eta}(\mathcal{S} \backslash \{g\})$. A non-empty, finite, convex subset of $\bb{Z}^2$ for which each vertex is $\eta$-generated is called an \emph{$\eta$-generating set}.
\end{definition}

A point $u_0 \in \mathcal{S}$ is $\eta$-generated by $\mathcal{S}$ if and only if, for each \(g \in \bb{Z}^2\), $\eta_{g-u}$ for $u \in \mathcal{S}$, with $u \neq u_0$, determines $\eta_{g-u_0}$.

Let $\mathcal{S} \subset \bb{Z}^2$ be an \(\eta\)-generating set. Only oriented lines parallels to some edge of \(\mathcal{S}\) may be one-sided nonexpansive directions on $\overline{Orb \, (\eta)}$:

\begin{lemma}[Cyr and Kra \cite{van}]\label{lem_genset_noedge_expas}
Let $\eta \in \mathcal{A}^{\bb{Z}^2}$\! and suppose $\pmb{\ell} \subset \bb{R}^2$ is an oriented line through the ori\-gin and $\mathcal{S} \subset \bb{Z}^2$ is a finite, convex set such that $\mathcal{S} \cap \pmb{\ell}_{\mathcal{S}} = \{g_0\}$ is $\eta$-generated by $\mathcal{S}$. Then $\pmb{\ell} \not\in \nexpd(\eta)$.
\end{lemma}

Let $\varphi(X) = \sum_{i} a_iX^{u_i}$, with $a_i \in R$ and $u_i \in \bb{Z}^2$, be a Laurent polynomial. The \emph{support} of $\varphi$ is defined as $$\supp(\varphi) := \{u_i \in \bb{Z}^2 : a_i \neq 0\}.$$ The \emph{reflected convex support of $\varphi$} is defined as $\mathcal{S}_{\varphi} := \conv(-\supp(\varphi)) \cap \bb{Z}^2$. 

\begin{remark}\label{rem_edge_convexsupport}
For $\varphi(X) = (X^{h_1}-1) \cdots (X^{h_m}-1)$, with \(h_1, \ldots, h_m \in (\bb{Z}^2)^*\), it is easy to see that \[\supp(\varphi) = \{(0,0)\} \cup \{h_{i_1}+\cdots+h_{i_r} : 1 \leq i_1 < \cdots < i_r \leq m, \ 1 \leq r \leq m\}.\] Suppose that \(\conv(\mathcal{S}_{\varphi})\) has positive area. Then every edge of \(\mathcal{S}_{\varphi}\) is either parallel or antiparallel to some vector \(h_i\), with \(1 \leq i \leq m\). If we suppose in addition that the vectors \(h_1, \ldots, h_m\) are in pairwise distinct directions, then, for each \(1 \leq i \leq m\), there exist edges \(-w_i,w_i \in E(\mathcal{S}_{\varphi})\) such that \(-w_i\) is parallel to \(-h_i\) and \(w_i\) is parallel to \(h_i\).  
\end{remark}	

\begin{lemma}[Szabados \cite{Szabados}]\label{lemma_supportgenerating}
Let $\eta \in \mathcal{A}^{\bb{Z}^2}$\!, with \(\mathcal{A} \subset R\), be a configuration. If $\varphi \in \ann_R(\eta)$, then $\mathcal{S}_{\varphi}$ is an $\eta$-generating set.	
\end{lemma}
\begin{proof}
Let $\varphi(X) = \sum_{i} a_iX^{u_i}$, with $a_i \in R$ and $u_i \in \bb{Z}^2$. Given $g \in \bb{Z}^2$ and $u_0 \in \supp(\varphi)$, the values $\eta_{g-u}$ for $u \in \supp(\varphi)$, with $u \neq u_0$, determines $\eta_{g-u_0}$ by applying equation \(a_1\eta_{g-u_1}+\cdots+a_n\eta_{g-u_n} = 0\). This means that $-u_0 \in -\supp(\varphi)$ is $\eta$-generated by $-\supp(\varphi)$.  
\end{proof}

As an immediate application of the previous two lemmas, we have the following remark:

\begin{remark}[Szabados \cite{Szabados}]\label{rem_edges_and_one-sided_direc}
Let $\eta \in \mathcal{A}^{\bb{Z}^2}$, with \(\mathcal{A} \subset R\), be a configuration. Sup\-pose that $\varphi(X) = (X^{h_1}-1) \cdots (X^{h_m}-1)\in \ann_R(\eta)$, with \(h_1, \ldots, h_m \in (\bb{Z}^2)^*\), and that \(\conv(\mathcal{S}_{\varphi})\) has positive area. In particular, \(\mathcal{S}_{\varphi}\) is an \(\eta\)-generating set. If $\pmb{\ell} \in \nexpd(\eta)$, then, thanks to Lem\-ma~\ref{lem_genset_noedge_expas}, \(\pmb{\ell}\) is parallel to some edge of \(\mathcal{S}_{\varphi}\). Hence, according to Re\-mark~\ref{rem_edge_convexsupport}, \(\pmb{\ell}\) is either parallel or an\-ti\-par\-al\-lel to some vector \(h_i\), with \(1 \leq i \leq m\).  
\end{remark}

The previous remark allows us to show that the conversely of Szabados's conjecture holds:

\begin{remark}\label{rem_converSazabados'sconj}
If $\eta = \eta_1+\cdots+\eta_m$ is a $R$-mi\-nimal periodic decomposition and $\ell \in \nexpl(\eta)$, then \(\ell\) contains a period of some $\eta_i$, with $1 \leq i \leq m$. Indeed, if \(m = 1\), then there is nothing to argue. Otherwise, let $h_i \in \bb{Z}^2$ be a period for $\eta_i$, with \(1 \leq i \leq m\), and consider the Laurent polynomial $\varphi(X) = (X^{h_1}-1) \cdots (X^{h_m}-1)$. Since the vectors \(h_1, \ldots, h_m\) are in pairwise distinct directions, then \(\conv(\mathcal{S}_{\varphi})\) has positive area. Since \(\varphi \in \ann_{R}(\eta)\) and, for some orientation, \(\pmb{\ell} \in \nexpd(\eta)\), by the previous remark it follows that \(\pmb{\ell}\) is either parallel or an\-ti\-par\-al\-lel to some vector \(h_i\), with \(1 \leq i \leq m\), which means that \(\ell\) contains a period of some $\eta_i$, with $1 \leq i \leq m$.   
\end{remark}

The next result is a rewrite of a particular case of \cite[Lemma 5.10]{frankskra}, a well-known result.

\begin{proposition}[Franks and Kra \cite{frankskra}]\label{prop_expan_period}
Let $\eta \in \mathcal{A}^{\bb{Z}^2}$, with \(\mathcal{A} \subset R\), be a configur- ation. Suppose \(\varphi(X) = (X^{h_1}-1)(X^{h_2}-1) \in \ann_R(\eta)\) and \(h_1,h_2 \in \bb{Z}^2\) are vectors in distinct directions. If \(\pmb{\ell} \subset \bb{R}^2\) is an oriented line through the origin parallel to an edge of \(\mathcal{S}_{\mathcal{\varphi}}\) and \(\pmb{\ell} \not\in \nexpd(\eta)\), then \(\eta\) is periodic.
\end{proposition}

As an immediate application of Proposition \ref{prop_expan_period}, we get that Szabados's conjecture holds for not fully periodic configurations with \(\bb{Z}\)-order \(2\).

To conclude this section, we will state Proposition \ref{prop_geom_convexset}, which will be essential in both cases \(R = \bb{Z}\) and \(R = \bb{Z}_p\). The case where \(R = \bb{Z}\) is an immediate consequence of Corol\-lary 24 of \cite{KariSzabados}. To provide a proof for the case where \(R = \bb{Z}_p\) we need first to prove the following lemma:

\begin{lemma}[Kari and Szabados \cite{kari}]\label{lem_existmulconf}
Let $\theta \in R[[X^{\pm 1}]]$ be a periodic configuration of period $u \in \bb{Z}^2$. Then, for a vector $v \in \bb{Z}^2$ linearly independent to $u$ over $\bb{R}^2$, there exists a periodic configuration $\vartheta \in R[[X^{\pm 1}]]$ of period $u$ such that $(X^v-1)\vartheta = \theta$.  
\end{lemma}
\begin{proof}
Recall that $(X^v-1)\vartheta = \theta$ means that $\vartheta_{g-v}-\vartheta_{g} = \theta_g$ for all $g \in \bb{Z}^2$. Let $\pmb{\ell}_0 \subset \bb{R}^2$ be an oriented line through the origin parallel to the vector $v$. For $g_0 \in \pmb{\ell}_0 \cap \bb{Z}^2$, let $g_1,\ldots,g_l \in \pmb{\ell}_0 \cap \bb{Z}^2$ be the points of $\bb{Z}^2$ between $g_0$ and $g_0+v$, but different from \(g_0\) and $g_0+v$. For each $0 \leq i \leq l$, we define $\vartheta_{g_i+tv} := \sum_{k=1}^{|t|} \theta_{g_i-(k-1)v}$ for $t < 0$, $\vartheta_{g_i} := 0$ and $\vartheta_{g_i+tv} := -\sum_{k=1}^t \theta_{g_i+kv}$ for $t > 0$. 

Let \(\pmb{\ell}_1, \ldots, \pmb{\ell}_n \subset \bb{R}^2\) be the oriented lines parallels to $v$ between $\pmb{\ell}_0$ and $\pmb{\ell}_0+u$, but different from \(\pmb{\ell}_0\) and $\pmb{\ell}_0+u$, that contain a point of \(\bb{Z}^2\). Proceeding as before, we define \(\vartheta_g\) for all \(g \in S := (\pmb{\ell}_0 \cup \pmb{\ell}_1 \cup \cdots \cup \pmb{\ell}_n) \cap \bb{Z}^2\). For any \(g \in \bb{Z}^2 \backslash S\), we define \(\vartheta_g = \vartheta_{g+tu}\), where  \(t \in \bb{Z}\) is the unique integer such that \(g+tu \in S\). By construction, $\vartheta$ is periodic of period \(u\) and, since $\theta$ is periodic of period $u$, $\vartheta$ satisfies all required conditions.  
\end{proof}  

\begin{proposition}[Kari and Szabados \cite{KariSzabados}]\label{prop_geom_convexset}
Let $\eta \in \mathcal{A}^{\bb{Z}^2}$, with $\mathcal{A} \subset R$, be a non-\linebreak periodic configuration. Suppose $\eta = \eta_1+\cdots+\eta_m$ is a $R$-minimal periodic decomposition, $h_i \in \bb{Z}^2$ is a period for $\eta_i$, with \(1 \leq i \leq m\), and $\varphi(X) = (X^{h_1}-1) \cdots (X^{h_{m}}-1)$. If $\psi \in \ann_{R}(\eta)$, then, for every edge $w \in E(\mathcal{S}_{\varphi})$, there exists an edge $w' \in E(\mathcal{S}_{\psi})$ parallel to $w$.
\end{proposition}
\begin{proof}
Let $\psi \in \ann_{R}(\eta)$ and suppose, by contradiction, that there exists \(w \in E(\mathcal{S}_{\varphi})\) that is not par\-al\-lel to any edge of \(\mathcal{S}_{\psi}\). According to Remark \ref{rem_edge_convexsupport}, \(w\) is either parallel or antiparallel to some vector \(h_{\iota}\), with \(1 \leq \iota \leq m\). Renaming \(\eta_1, \ldots, \eta_m\) if necessary, we may assume that \(\iota = 1\). Let \(\pmb{\ell} \subset \bb{R}^2\) be the oriented line through the origin parallel to \(w\). In particular, $\mathcal{S}_{\psi} \cap \pmb{\ell}_{\mathcal{S}_{\psi}}$ is a vertex of $\mathcal{S}_{\psi}$. Since $\mathcal{S}_{\psi}$ is an \(\eta\)-generating set, Lem\-ma~\ref{lem_genset_noedge_expas} implies that \(\pmb{\ell} \not\in \nexpd(\eta)\). Hence, by Proposition \ref{prop_expan_period} it follows that \(m \geq 3\). Thus, let \[\sigma(X) := (X^{h_1}-1)(X^{h_2}-1) \ \ \text{and} \ \ \varphi_3(X) := (X^{h_3}-1) \cdots (X^{h_m}-1).\] Note that \(\varphi_3\eta = \varphi_3\eta_1+\varphi_3\eta_2\) and \(\sigma \in \ann_{R}(\varphi_3\eta)\). Moreover, \(\mathcal{S}_{\sigma}\) has an edge \(w' \in E(\mathcal{S}_{\sigma})\) parallel to \(w\). Since \[\psi(\varphi_3\eta) = \varphi_3(\psi\eta) = 0,\] i.e., \(\psi \in \ann_{R}(\varphi_3\eta)\), then, as before, Lem\-ma~\ref{lem_genset_noedge_expas} implies that \(\pmb{\ell} \not\in \nexpd(\varphi_3\eta)\). Hence, being \(\pmb{\ell}\) parallel to \(w\) and so to \(w'\), by Proposition \ref{prop_expan_period} results that  \(\varphi_3\eta\) is periodic. 

We claim that \(\varphi_3\eta\) has a period \(u \in \bb{Z}^2\) such that the vectors \(u,h_3,h_4, \ldots, h_m\) are in pairwise distinct directions. Indeed, if \(\varphi_3\eta\) is a fully periodic configuration,\linebreak then, we may consider, for instance, a vector \(u \in \bb{Z}^2\) parallel to \(h_1\) or \(h_2\). Otherwise, the unique line \(\ell' \in \nexpl(\varphi_3\eta)\) contains any period for \(\varphi_3\eta\). Since, for some orientation, \(\pmb{\ell}' \in \nexpd(\varphi_3\eta)\), then, according to Remark \ref{rem_edges_and_one-sided_direc}, \(\pmb{\ell}'\) is either parallel or antiparallel to \(h_2\), which proves the claim.

Let $u \in \bb{Z}^2$ be a period for $\varphi_3\eta$ as in the previous claim. To conclude the proof, we will contradict the minimality of $m$. We claim that $\eta$ can be decomposed into a sum of $m-1$ periodic configurations. Indeed, by Lemma~\ref{lem_existmulconf} there exists a periodic configuration $\Theta_2 \in R[[X^{\pm 1}]]$, with $u$ a period for $\Theta_2$, such that $(X^{h_3}-1)\Theta_2 = \varphi_3\eta$. For each \(3 \leq i \leq m\), let $$\varphi_n(X) := \prod_{i = n}^{m} (X^{h_i}-1).$$ For $\Theta_3 := \varphi_4\eta-\Theta_2$, it follows that $\varphi_4\eta = \Theta_2+\Theta_3$ and $$(X^{h_3}-1)\Theta_3 = \varphi_3\eta - (X^{h_3}-1)\Theta_2 = 0.$$ Suppose we have constructed periodic configurations $\theta_2,\ldots,\theta_{m-1} \in R[[X^{\pm 1}]]$, with $u$ a period for $\theta_2$ and, for each $3 \leq i \leq m-1$, $h_i$ a period for $\theta_i$, such that $$(X^{m}-1)\eta = \theta_2+\cdots+\theta_{m-1}.$$ Thanks to Lemma~\ref{lem_existmulconf}, there exist periodic configurations $\varphi_2,\ldots,\varphi_{m-1} \in R[[X^{\pm 1}]]$, with $u$ a period for $\varphi_2$ and, for each $3 \leq i \leq m-1$, $h_i$ a period for $\varphi_i$, such that $(X^{m}-1)\varphi_i = \theta_i$. For $\varphi_{m} := \eta - \varphi_2-\cdots-\varphi_{m-1}$, it follows that $\eta = \varphi_2+\cdots+\varphi_{m}$ and $$(X^{m}-1)\varphi_{m} = (X^{m}-1)\eta -(X^{m}-1)\varphi_2 - \cdots - (X^{m}-1)\varphi_{m-1} = 0,$$ which proves the claim and contradicts the minimality of $m$.
\end{proof}

\section{Proof of main results}
\label{sec3}

\begin{lemma}\label{lemma_fullyperiodic}
Let $\vartheta \in \mathcal{A}^{\bb{Z}^2}$\!, with $\mathcal{A} \subset \bb{Z}_p$, and suppose \(\vartheta = \vartheta_1+\cdots+\vartheta_m\) is a \(\bb{Z}_p\)-pe-\linebreak riodic decomposition and \(h_i \in \bb{Z}^2\) is a period for \(\vartheta_i\), with \(1 \leq i \leq m\). If \(\vartheta\) is non-periodic, suppose in addition that the vectors \(h_1, \ldots, h_m\) are in pairwise distinct directions. If \(\vartheta = \vartheta'_1+\cdots+\vartheta'_n\), with \(n < m\), is a \(\bb{Z}_p\)-minimal periodic decomposition, then the following conditions hold.
\begin{enumerate}[(i)]\setlength{\itemsep}{6pt}
	\item  There exists a sequence \(1 \leq i_1 < \cdots < i_n \leq m\) where, for each \(1 \leq j \leq n\), there exists \(t_j \in \bb{Z}\) such that \(u'_j = t_jh_{i_j}\) is a period for \(\vartheta'_j\); 
	
	\item If \(\vartheta\) is fully periodic, then \(\vartheta = \vartheta'_1\) and condition (i) holds for every \(1 \leq i_1 \leq m\);
		
	\item For any \(1 \leq i \leq m\), with \(i \neq i_1, \ldots, i_n\), \(\vartheta_i\) is fully periodic. 
\end{enumerate}
\end{lemma}
\begin{proof}
If \(\vartheta\) is periodic of period \(u \in \bb{Z}^2\), but not fully periodic, then, the unique line \(\ell \in \nexpl(\vartheta)\) contains \(u\). Since, for some orientation, \(\pmb{\ell} \in \nexpd(\vartheta)\), then, according to Remark~\ref{rem_edges_and_one-sided_direc}, \(\pmb{\ell}\) and so \(u\) is either parallel or antiparallel to some vector \(h_i\), with \(1 \leq i \leq m\). Thus, for \(i_1 = i\), there exist \(t_1 \in \bb{Z}\) and \(s_1 \in \bb{Z}\) such that \(s_1u = t_1h_{i_1}\). We finish this case by setting \(u'_1 = s_1u\). If \(\vartheta\) is non-periodic, let \(u_j \in \bb{Z}^2\) be a period for \(\vartheta'_j\), with \(1 \leq j \leq n\). Since \[(X^{u_1}-1) \cdots (X^{u_n}-1), (X^{h_1}-1) \cdots (X^{h_m}-1) \in \ann_{\bb{Z}_p}(\vartheta),\] from Proposition~\ref{prop_geom_convexset} we get that each vector \(u_j\) is either parallel or antiparallel to some vector \(h_i\). Thus, there exist  a sequence \(1 \leq i_1 < \cdots < i_n \leq m\) such that, for each \(1 \leq j \leq n\), \(s_ju_j = t_jh_{i_j}\) holds for some \(t_j,s_j \in \bb{Z}\). We establish this case by setting \(u'_j = s_ju_j\) for every \(1 \leq j \leq n\). 

If \(\vartheta\) is fully periodic, then \(\vartheta_1'\) is fully periodic. Hence, for any \(1 \leq i_1 \leq m\), there exists \(t_1 \in \bb{Z}\) such that \(u'_1 = t_1h_{i_1}\) is a period for \(\vartheta_1'\), which proves conditions~(i) and (ii).\medbreak

To approach condition (iii), consider \(A = \{1 \leq i \leq m : i \neq i_1, \ldots, i_n\}\) and write \(\theta = \sum_{i \in A} \vartheta_i\). Note that \[0 = (X^{u'_1}-1) \cdots (X^{u'_n}-1)\vartheta = (X^{t_1h_{i_1}}-1) \cdots (X^{t_nh_{i_n}}-1)\theta,\] i.e., \[\varphi(X) := (X^{t_1h_{i_1}}-1) \cdots (X^{t_nh_{i_n}}-1) \in \ann_{\bb{Z}_p}(\theta).\] Clearly, \(\psi(X) := \prod_{i \in A} (X^{h_i}-1) \in \ann_{\bb{Z}_p}(\theta)\). By Lemma \ref{lem_genset_noedge_expas}, if \(\pmb{\ell}' \in \nexpd(\theta)\), then \(\pmb{\ell}'\) must be parallel to an edge of \(\mathcal{S}_{\varphi}\) and parallel to an edge of \(\mathcal{S}_{\psi}\). Since the vectors \(h_1, \ldots, h_m\) are in pairwise distinct directions,  \(\mathcal{S}_{\varphi}\) and \(\mathcal{S}_{\psi}\) does not have parallel edges, which means that \(\nexpd(\theta) = \emptyset\). Hence, \(\nexpl(\theta) = \emptyset\) and, by Theorem~\ref{corollary_Boyle-Lind_tmh}, we get that \(\theta\) is fully periodic. Let \(l_1, \ldots, l_k\), with \(k = |A|\), be an enumeration of the elements of \(A\). If \(k = 1\), there is nothing to argue. Otherwise, as \(\theta\) is fully periodic, then \(\theta-\vartheta_{l_1}\) is periodic with period parallel to \(h_{l_1}\). Let \(u''_1 = t_{l_1}h_{l_1}\), with \(t_{l_1} \in \bb{Z}\), be a period for \(\theta-\vartheta_{l_1}\). Note that \[(X^{u''_1}-1), \psi_1(X) := \prod_{\substack{i \in A\\i \neq l_1}} (X^{h_i}-1) \in \ann_{\bb{Z}_p}(\theta-\vartheta_{l_1}).\] By Lemma \ref{lem_genset_noedge_expas}, if \(\pmb{\ell}' \in \nexpd(\theta-\vartheta_{l_1})\), then we must have \(u''_1 \in \pmb{\ell}'\) and \(\pmb{\ell}'\) must be parallel to an edge of \(\mathcal{S}_{\psi_1}\). So as before, \(\nexpl(\theta-\vartheta_{l_1}) = \emptyset\) and, by Theorem~\ref{corollary_Boyle-Lind_tmh}, we get that \(\theta-\vartheta_{l_1}\) is fully periodic. Proceeding this way, suppose we have showed that \[\theta-\vartheta_{l_1}-\cdots-\vartheta_{l_{k-2}} = \vartheta_{l_{k-1}}+\vartheta_{l_k}\] is fully periodic. Hence, \(\theta-\vartheta_{l_1}-\cdots-\vartheta_{l_{k-1}} = \vartheta_{l_k}\) is periodic with period parallel to \(h_{l_{k-1}}\). Being \(\vartheta_{l_k}\) periodic of period \(h_{l_k}\), we get that \(\vartheta_{l_k}\) is fully periodic. Therefore, as \(l_1, \ldots, l_k\) is an arbitrary enumeration of the elements of \(A\), then, for any \(1 \leq i \leq m\), with \(i \neq i_1, \ldots, i_n\), \(\vartheta_i\) is fully periodic, which establishes condition (iii) and proves the lemma.
\end{proof}

Let $\eta \in \mathcal{A}^{\bb{Z}^2}$, with \(\mathcal{A} \subset \bb{Z}_+\), and suppose \(\eta = \eta_1+\cdots+\eta_m\) is a \(\bb{Z}\)-minimal periodic decomposition.  For a prime number \(p \in \bb{N}\), we may consider the \(\bb{Z}_p\)-perio- dic decomposition \(\overline{\eta} = \overline{\eta}_1+\cdots+\overline{\eta}_m\), where the bar denotes the congruence modulo \(p\). If \(p\) is large enough so that \(\mathcal{A} \subset [[p]]\), then \(\eta\) and \(\overline{\eta}\) are the same configurations, i.e., the bar does not change \(\eta\). However, each configuration \(\overline{\eta}_i\) is defined on the finite alphabet \(\bb{Z}_p\). Furthermore, if \(\psi(X) = \sum_{i = 1}^{n} a_iX^{u_i}\), with \(a_i \in \bb{Z}\) and \(u_i \in \bb{Z}^2\), annihilates \(\eta\), then by (\ref{defproduct}) we get that \[\overline{a_1}\cdot\overline{\eta}_{g-u_1}+\cdots+\overline{a_n} \cdot \overline{\eta}_{g-u_n} = 0 \quad \forall g \in \bb{Z}^2,\] that is, \[\overline{\psi}(X) = \sum_{i = 1}^{n} \overline{a_i}X^{u_i} \in \ann_{\bb{Z}_p}(\overline{\eta}).\]

\subsection{Proof of Theorem \ref{mainthm1}}
If \(m = 1\), the result holds trivially. Otherwise, for each \(1 \leq i \leq m\), let \(h_i \in \bb{Z}^2\) be a period for \(\eta_i\) and let \[\varphi_i(X) = \prod_{\substack{k=1\\k \neq i}}^{m} (X^{h_k}-1).\] Note that each configuration \(\varphi_i\eta = \varphi_i\eta_i\) is defined on a finite alphabet. Thus, for each \(1 \leq i \leq m\), let \(N_i \in \bb{N}\) be such that \[(\varphi_i\eta+N_i1\hspace{-0.13cm}1)_g \in \bb{Z}_+ \quad \forall \ g \in \bb{Z}^2,\] where \(1\hspace{-0.13cm}1_g = 1\) for all \(g \in \bb{Z}^2\). Let \(p \in \bb{N}\), with \(\mathcal{A} \subset [[p]]\), be a prime number large enough so that, for any \(1 \leq i \leq m\), \[\{(\varphi_i\eta_i+N_i1\hspace{-0.13cm}1)_g : g \in \bb{Z}^2\} \subset [[p]].\] 

We claim that, for each \(1 \leq i \leq m\), \(\overline{\eta}_i\) is not fully periodic. Indeed, suppose, by contradiction, that \(\overline{\eta}_i\) is fully periodic. Since \[\overline{\varphi_i\eta+N_i1\hspace{-0.13cm}1} = \overline{\varphi_i\eta_i+N_i1\hspace{-0.13cm}1} = \varphi_i\overline{\eta}_i+\overline{N_i1\hspace{-0.13cm}1},\] where on the right side \(\varphi_i \in \bb{Z}_p[X^{\pm 1}]\), then \(\overline{\varphi_i\eta+N_i1\hspace{-0.13cm}1}\) is fully periodic. However, as  \[\{(\varphi_i\eta_i+N_i1\hspace{-0.13cm}1)_g : g \in \bb{Z}^2\} \subset [[p]],\] the bar does not change \(\varphi_i\eta+N_i1\hspace{-0.13cm}1\), which means that \(\varphi_i\eta+N_i1\hspace{-0.13cm}1\) and so \(\varphi_i\eta\) is fully periodic. Consider \(1 \leq j \leq m\), with \(j \neq i\), and let \(h'_j = t_{j}h_{j}\), with \(t_{j} \in \bb{Z}\), be a period for \(\varphi_i\eta\). Then \[(X^{h'_j}-1)\varphi_i\eta = 0,\] i.e., \(\psi(X) := (X^{h'_j}-1)\varphi_i(X) \in \ann_{\bb{Z}}(\eta)\). Since $\mathcal{S}_{\psi}$ does not have any edge parallel to $h_i$, by Pro\-po\-si\-tion~\ref{prop_geom_convexset} we reach a contradiction, which establishes the claim.\medbreak

We claim that \(\overline{\eta} = \overline{\eta}_1+\cdots+\overline{\eta}_m\) is a \(\bb{Z}_p\)-mini\-mal periodic decomposition. Indeed, if \(\overline{\eta} = \eta'_1+\cdots+\eta'_n\), with \(n < m\), is a \(\bb{Z}_p\)-mi\-ni\-mal periodic decomposition, then, according to Lemma \ref{lemma_fullyperiodic}, there exists a sequence \(1 \leq i_1 < \cdots < i_n \leq m\) such that, for each \(1 \leq i \leq m\), with \(i \neq i_1, \ldots, i_n\), \(\overline{\eta}_i\) is fully periodic, which due to the previous claim is a contradiction.\hfill\(\Box\)\medbreak                        

Let $\eta \in \mathcal{A}^{\bb{Z}^2}$ be a configuration and suppose $\mathcal{U} \subset \bb{Z}^2$ is a non-empty, convex set. We say that \(\eta\sob{\mathcal{U}}\) is \emph{periodic} if there exists a vector $h \in (\bb{Z}^2)^*$, called \emph{period of \(\eta\sob{\mathcal{U}}\)},\linebreak such that $$\eta_{g-h} = \eta_{g} \quad \forall g \in \mathcal{U} \cap (\mathcal{U}+h).$$ If \(\eta\sob{\mathcal{U}}\) has two periods linearly independents over \(\bb{R}^2\), we say that \(\eta\sob{\mathcal{U}}\) is \emph{fully peri- odic}.	

\begin{theorem}\label{mainthm_finitealphabet}
Let $\eta \in \mathcal{A}^{\bb{Z}^2}$, with \(\mathcal{A} \subset \bb{Z}_p\), be a not fully periodic configuration and suppose \(\eta = \eta_1+\cdots+\eta_m\) is a \(\bb{Z}_p\)-minimal periodic decomposition. If for any configuration \(x \in \overline{Orb \, (\eta)}\) the subshift \(\overline{Orb \, (x)}\) has a periodic configuration, then, for each $1 \leq i \leq m$, a line $\ell \subset \bb{R}^2$ contains a period for $\eta_i$ if and only if $\ell \in \nexpl(\eta)$.
\end{theorem}
\begin{proof}
The proof will be done by induction on \(m\). If \(m = 1\), the theorem holds trivially. If \(m = 2\), the theorem follows by Proposition~\ref{prop_expan_period}. Thus we may assume \(m \geq 3\). Suppose, by induction hypothesis, that Theorem~\ref{mainthm_finitealphabet} holds for any not fully periodic configuration with \(\bb{Z}_p\)-order \(n \leq m-1\). 

Let \(h_i \in \bb{Z}^2\) be a period for \(\eta_i\), with \(1 \leq i \leq m\). Let \(\vartheta \in \overline{Orb \, (\eta)}\) be a periodic configuration and suppose \((g_n)_{n \in \bb{N}} \subset \bb{Z}^2\) is a sequence such that \[\displaystyle\lim_{n \to +\infty} T^{g_n}\eta = \vartheta.\] By passing to a subsequence, we may assume that, for any \(1 \leq i \leq m\), \[\displaystyle\lim_{n \to +\infty} T^{g_{n}}\eta_i = \vartheta_i.\] Since \(T^{g_n}\eta = T^{g_n}\eta_1+\cdots+ T^{g_n}\eta_m\), then \(\vartheta = \vartheta_1+\cdots+\vartheta_m\) is a \(\bb{Z}_p\)-periodic decomposition, where \(\vartheta_i\) is periodic of period \(h_i\), with \(1 \leq i \leq m\). If we denote \(\vartheta = \vartheta'_1\), then \(\vartheta = \vartheta'_1\), with \(1 < m\), is a \(\bb{Z}_p\)-minimal periodic de\-com\-po\-si\-tion. Thus, according to Lemma~\ref{lemma_fullyperiodic}, there exists \(1 \leq \iota_1 \leq m\) such that, for any \(1 \leq i \leq m\), with \(i \neq \iota_1\), \(\vartheta_i\) is fully periodic. 

Given \(1 \leq k \leq m\), with \(k \neq \iota_1\), let \(1 \leq j \leq m\), with \(j \neq \iota_1\), be such that \(j \neq k\). Let \(u_j = t_jh_j\), with \(t_j \in \bb{Z}\), be a period for \(\vartheta_k\) and consider \[R_{l} := \left\{tu_j+sh_k \in \bb{Z}^2: t,s \in [-l,l]\right\} \quad (l \in \bb{N}).\] For each \(l \in \bb{N}\), there exists \(n_l \in \bb{N}\) such that, for any \(1 \leq i \leq m\), with \(i \neq \iota_1\), \[(T^{g_{n_l}}\eta_{i})\sob{R_{l}} = \vartheta_i\sob{R_{l}}.\] Let \(\ell_j,\ell_k \subset \bb{R}^2\) be the lines through the origin such that \(u_j \in \ell_j\) and \(h_k \in \ell_k\) and consider the strips \[S_{\ell_j}(l) := \{g+tu_j \in \bb{Z}^2 : g \in R_{l}, \ t \in \bb{Z}\}\] and \[S_{\ell_k}(l) := \{g+th_k \in \bb{Z}^2 : g \in R_{l}, \ t \in \bb{Z}\}.\] Note that \((T^{g_{n_l}}\eta_k)\sob{S_{\ell_k}(l)}\) is fully periodic of periods \(u_j\) and \(h_k\). Being \(\eta_k\) not fully periodic, we may consider the largest convex set \(L_{\ell_k}(l) \subset \bb{Z}^2\), with \(S_{\ell_k}(l) \subset L_{\ell_k}(l)\), such that
\begin{enumerate}[(i)]\setlength{\itemsep}{6pt}
	\item \((T^{g_{n_l}}\eta_k)\sob{L_{\ell_k}(l)}\) is fully periodic of periods \(u_j\) and \(h_k\);
	\item \(g \in L_{\ell_k}(l)\) if and only if \((g+\ell_k) \cap \bb{Z}^2 \subset L_{\ell_k}(l)\).
\end{enumerate}

\vspace{10pt}

\noindent Of course, \(L_{\ell_k}(l)\) is either a strip or a half plane, which implies that \(L_{\ell_k}(l)\) has at least one edge \(w_{\ell_k}(l) \in E(L_{\ell_k}(l))\). Let \(s_{l} \in \bb{Z}\) be such that \(-s_lu_j\) is closest as possible from \(w_{\ell_k}(l)\) (see Figure \ref{fig1}).
\begin{figure}[h!]
	\centering\def\svgwidth{7.4cm}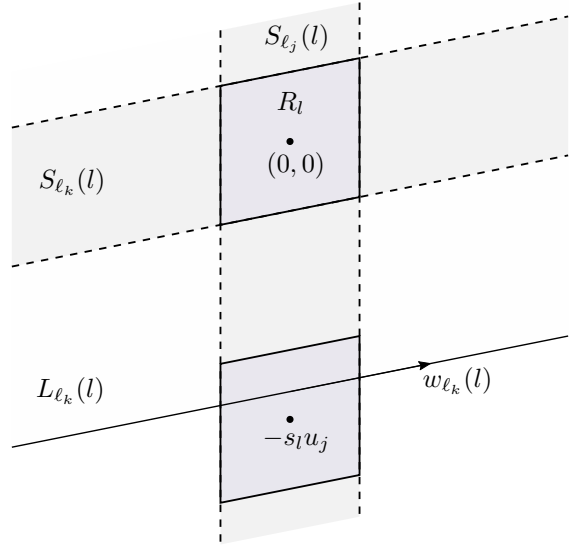
	\caption{The strips \(S_{\ell_k}(l)\) and \(S_{\ell_j}(l)\) and the set \(L_{\ell_k}(l)\).}
	\label{fig1}
\end{figure}
By passing to a subsequence, we may assume that \[\lim_{l \to +\infty} T^{g_{n_l}+s_lu_j}\eta = \varTheta\] and, for each \(1 \leq i \leq m\), \[\lim_{l \to +\infty} T^{g_{n_l}+s_lu_j}\eta_i = \varTheta_i.\] Since \(T^{g_{n_l}+s_lu_j}\eta = T^{g_{n_l}+s_lu_j}\eta_1+\cdots+T^{g_{n_l}+s_lu_j}\eta_m\), then \(\varTheta = \varTheta_1+\cdots+\varTheta_m\) is a \(\bb{Z}_p\)-periodic decomposition, where \(h_i\) is a period for \(\varTheta_i\), with \(1 \leq i \leq m\). By construction, \(\varTheta_j\)  is fully periodic, which implies that \(\varTheta\) has \(\bb{Z}_p\)-order lower than \(m\). Let \(\varTheta = \varTheta'_1+\cdots+\varTheta'_n\), with \(n \leq m-1\), be a \(\bb{Z}_p\)-minimal periodic decomposition. Due to Lemma \ref{lemma_fullyperiodic}, there exists a sequence \(1 \leq i_1 < \cdots < i_n \leq m\) such that, for each \(1 \leq i\ \leq m\), with \(i \neq i_1,\ldots, i_n\), \(\varTheta_i\) is fully periodic. 

\begin{claim}
\(k \in \{i_1, \ldots, i_n\}\).
\end{claim}
Indeed, if \(k \neq i_1,\ldots, i_n\), then, by the previous discussion, \(\varTheta_k\) is fully periodic. For \(r \in \bb{N}\) as large as we want, there exists \(l \in \bb{N}\), with \(l > r\), such that \[(T^{g_{n_l}}\eta_k)\sob{R_{r}-s_lu_j}\] is fully periodic. But, since \(L_{\ell_k}(l) \cap (R_{r}-s_lu_j)\) can be as large as we want, by item (i)  we get that \[(T^{g_{n_l}}\eta_k)\sob{L_{\ell_k}(l) \cup (S_{\ell_k}(r)-s_lu_j)}\] is fully periodic of periods \(u_j\) and \(h_k\), which is a contradiction and establishes the claim.\medbreak

Note that \(\varTheta\) can not be fully periodic, since, otherwise, \(\varTheta = \varTheta'_1\) and con\-di\-tion~(ii) of Lemma~\ref{lemma_fullyperiodic} allows us to consider \(i_1 \neq k \), which due to the previous claim is a contradiction. Thus, since \(\varTheta\) has \(\bb{Z}_p\)-order \(n \leq m-1\) and for any configuration \(x\) \(\in \overline{Orb \, (\varTheta)}\) the subshift \(\overline{Orb \, (x)}\) has a periodic configuration, by induction hypothesis Szabados's conjecture holds for \(\varTheta\). 

Let \(\ell_1, \ldots, \ell_{m} \subset \bb{R}^2\) be the lines through the origin such that, for each \(1 \leq i \leq m\), \(h_i \in \ell_i\).

\begin{claim}
\(\ell_1, \ldots, \ell_{m} \in \nexpl(\eta)\).
\end{claim}
Indeed, since \(k \in \{i_1, \ldots, i_n\}\), it follows by condition (i) of Lemma~\ref{lemma_fullyperiodic} that some configuration \(\varTheta'_i\), with \(1 \leq i \leq n\), has a period parallel to \(h_k\). Thus, as Szabados's conjecture holds for \(\varTheta\), then \(\ell_k \in \nexpl(\varTheta) \subset \nexpl(\eta)\). Being \(1 \leq k \leq m\), with \(k \neq \iota_1\), an arbitrary integer, we conclude that \(\ell_i \in \nexpl(\eta)\) for every \(1 \leq i \leq m\), with \(i \neq \iota_1\). 

If \(\vartheta\) is fully periodic, then by repeating the same construction for \(1 \leq l_1 \leq m\), with \(l_1 \neq \iota_1\) (see condition (ii) of Lemma~\ref{lemma_fullyperiodic}), we get that \(\ell_{\iota_1} \in \nexpl(\eta)\). 

If \(\vartheta = \vartheta'_1\) is not fully periodic, the unique line \(\ell' \in \nexpl(\vartheta'_1) \subset \nexpl(\eta)\) contains any period for \(\vartheta'_1\). Thanks to condition (i) of Lemma~\ref{lemma_fullyperiodic}, we see that such a period is either parallel or antiparallel to \(h_{\iota_1}\), which implies that \(\ell' = \ell_{\iota_1}\) and proves the claim.\medbreak

Therefore, the inductive step is done, which proves Theorem~\ref{mainthm_finitealphabet}.
\end{proof}

\subsection{Proof of Theorem \ref{mainthm2}} Initially, changing the alphabet if necessary, we may assume that \(\mathcal{A} \subset \bb{Z}_+\). Thanks to Theorem \ref{mainthm1}, there exist a prime \(p \in \bb{N}\), with \(\mathcal{A} \subset [[p]]\), such that \(\overline{\eta} = \overline{\eta}_1+\cdots+\overline{\eta}_m\) is a \(\bb{Z}_p\)-minimal periodic decomposition. Since $\ell \in \nexpl(\eta)$ if and only if $\ell \in \nexpl(\overline{\eta})$, the result follows by Theorem \ref{mainthm_finitealphabet}.\hfill\(\Box\).\medbreak

\end{document}

%% file: fig1.pdf_tex
\begingroup%
  \makeatletter%
  \providecommand\color[2][]{%
    \errmessage{(Inkscape) Color is used for the text in Inkscape, but the package 'color.sty' is not loaded}%
    \renewcommand\color[2][]{}%
  }%
  \providecommand\transparent[1]{%
    \errmessage{(Inkscape) Transparency is used (non-zero) for the text in Inkscape, but the package 'transparent.sty' is not loaded}%
    \renewcommand\transparent[1]{}%
  }%
  \providecommand\rotatebox[2]{#2}%
  \newcommand*\fsize{\dimexpr\f@size pt\relax}%
  \newcommand*\lineheight[1]{\fontsize{\fsize}{#1\fsize}\selectfont}%
  \ifx\svgwidth\undefined%
    \setlength{\unitlength}{450.29414606bp}%
    \ifx\svgscale\undefined%
      \relax%
    \else%
      \setlength{\unitlength}{\unitlength * \real{\svgscale}}%
    \fi%
  \else%
    \setlength{\unitlength}{\svgwidth}%
  \fi%
  \global\let\svgwidth\undefined%
  \global\let\svgscale\undefined%
  \makeatother%
  \begin{picture}(1,1.04931411)%
    \lineheight{1}%
    \setlength\tabcolsep{0pt}%
    \put(0,0){\includegraphics[width=\unitlength,page=1]{fig1.pdf}}%
    \put(0.7386337,0.28358005){\color[rgb]{0,0,0}\makebox(0,0)[lt]{\lineheight{1.25}\smash{\begin{tabular}[t]{l}$w_{\ell_k}(l)$\end{tabular}}}}%
    \put(0,0){\includegraphics[width=\unitlength,page=2]{fig1.pdf}}%
    \put(0.04387856,0.26450207){\color[rgb]{0,0,0}\makebox(0,0)[lt]{\lineheight{1.25}\smash{\begin{tabular}[t]{l}$L_{\ell_k}(l)$\end{tabular}}}}%
    \put(0.45376362,0.90024105){\color[rgb]{0,0,0}\makebox(0,0)[lt]{\lineheight{1.25}\smash{\begin{tabular}[t]{l}$S_{\ell_j}(l)$\end{tabular}}}}%
    \put(0.45768519,0.67043521){\color[rgb]{0,0,0}\makebox(0,0)[lt]{\lineheight{1.25}\smash{\begin{tabular}[t]{l}$(0,0)$\end{tabular}}}}%
    \put(0.04852919,0.64011114){\color[rgb]{0,0,0}\makebox(0,0)[lt]{\lineheight{1.25}\smash{\begin{tabular}[t]{l}$S_{\ell_k}(l)$\end{tabular}}}}%
    \put(0.45129211,0.17846222){\color[rgb]{0,0,0}\makebox(0,0)[lt]{\lineheight{1.25}\smash{\begin{tabular}[t]{l}$-s_lu_j$\end{tabular}}}}%
    \put(0.47629096,0.78038656){\color[rgb]{0,0,0}\makebox(0,0)[lt]{\lineheight{1.25}\smash{\begin{tabular}[t]{l}$R_l$\end{tabular}}}}%
  \end{picture}%
\endgroup%